\documentclass{article}

\usepackage{amssymb}
\usepackage{amsthm}
\usepackage{mathtools}
\usepackage{mathscinet}

\usepackage[top=2cm, bottom=2cm, left=2.5cm, right=2.5cm]{geometry}
\usepackage{hyperref}
\usepackage{authblk}

\newtheorem{theorem}{Theorem}[section]
\newtheorem{lemma}[theorem]{Lemma}

\theoremstyle{definition}
\newtheorem{defn}{Definition}[section]

\theoremstyle{remark}
\newtheorem{remark}{Remark}[section]

\newcommand\Sp{\operatorname{Sp}}
\newcommand\Id{\operatorname{Id}}

\begin{document}

\title{Second order asymptotics for Krein indefinite multipliers with multiplicity two}
\author{Yinshan Chang\thanks{College of Mathematics, Sichuan University, Chengdu 610065, China\\
Email: ychang@scu.edu.cn\\
Supported by NSFC Grant 11701395 and the Fundamental Research Funds for the Central Universities (No. YJ201661)}  and Jingzhi Yan\thanks{College of Mathematics, Sichuan University, Chengdu 610065, China\\
Email: jyan@scu.edu.cn\\
Supported by the Fundamental Research Funds for the Central Universities (No. YJ201660)}}

\date{}

\maketitle

\begin{abstract}
 We consider linear Hamiltonian equations in $\mathbb{R}^{4}$ of the following type
\begin{equation}
 \frac{\mathrm{d}\gamma}{\mathrm{d}t}(t)=J_{4}A(t)\gamma(t), \gamma(0)\in\Sp(4,\mathbb{R}),
\end{equation}
where $J=J_{4}\overset{\text{def}}{=}\begin{bmatrix}0 & \Id_2\\-\Id_2 & 0\end{bmatrix}$ and $A:t\mapsto A(t)$ is a $C^1$-continuous curve in the space of $4\times 4$ real matrices which are symmetric. We obtain second order asymptotics for the eigenvalues bifurcated from non-real Krein indefinite eigenvalues with multiplicity two.
\end{abstract}

\section{Introduction}

We consider linear Hamiltonian equations in $\mathbb{R}^{4}$ of the following type
\begin{equation}\label{eq: ham system R modified initial cont pos}
 \frac{\mathrm{d}\gamma}{\mathrm{d}t}(t)=J_{4}A(t)\gamma(t), \gamma(0)\in\Sp(4,\mathbb{R}),
\end{equation}
where $J=J_{4}\overset{\text{def}}{=}\begin{bmatrix}0 & \Id_2\\-\Id_2 & 0\end{bmatrix}$ and $A:t\mapsto A(t)$ is a $C^1$-continuous curve in the space of $4\times 4$ real matrices which are symmetric.

The system \eqref{eq: ham system R modified initial cont pos} arises naturally from perturbations of linearized Hamiltonian equations. Indeed, let $\varepsilon\in\mathbb{R}$ be a \emph{real} perturbation parameter. Consider
\begin{equation}\label{eq: perturbed ham system R}
 \frac{\partial\gamma}{\partial t}(t,\varepsilon)=J_{4}A(t,\varepsilon)\gamma(t,\varepsilon),\quad \gamma(0,\varepsilon)=\Id_{4},
\end{equation}
where
$\frac{\partial A}{\partial \varepsilon}(t,\varepsilon)$ is jointly continuous . Then, for fixed $T$, as $\varepsilon$ varies, the endpoint matrix $\gamma(T,\varepsilon)$ is a $C^{2}$-curve satisfying \eqref{eq: ham system R modified initial cont pos}. More precisely,
\begin{equation}\label{eq: equation for partial partial gamma T epsilon}
 \frac{\mathrm{\partial}}{\mathrm{\partial}\varepsilon}\gamma(T,\varepsilon)=J_{4}B(T,\varepsilon)\gamma(T,\varepsilon),
\end{equation}
where
\begin{align}\label{defn: BTepsilon}
 B(T,\varepsilon)=&-J_{4}\frac{\partial\gamma}{\partial\varepsilon}(T,\varepsilon)\gamma(T,\varepsilon)^{-1}\notag\\
 =&\int_{0}^{T}(\gamma(T,\varepsilon)^{-1})^{T}\gamma(t,\varepsilon)^{T}\frac{\partial}{\partial\varepsilon}A(t,\varepsilon)\gamma(t,\varepsilon)\gamma(T,\varepsilon)^{-1}\,\mathrm{d}t.
\end{align}
Please see \cite[Eq. (5) and Appendix A]{KLC1} for a proof.

Many people have studied the system \eqref{eq: ham system R modified initial cont pos} or \eqref{eq: perturbed ham system R}, see Ekeland \cite{EkelandMR1051888}, Yakubovich and Starzhinskii\cite{YakubovichStarzhinskiiMR0364740}, and the references there. We are interested in the bifurcation of a Krein indefinite eigenvalue with multiplicity two. The first order asymptotics and qualitative behavior of the bifurcated eigenvalues were firstly discovered by Krein and Lyubarskii in \cite{KreinLyubarskiiMR0142832} for the end-point matrix $\gamma(T,\varepsilon)$ of perturbed linear Hamiltonian equations in $\mathbb{R}^{2n}$ under certain positivity and linearity assumption on $A(t,\varepsilon)$. Recently, a general version of the Krein-Lyubarskii theorem was obtained by Chang, Long and Wang in \cite{KLC1} for $C^1$ paths of symplectic matrices (corresponding to the solution of \eqref{eq: ham system R modified initial cont pos}). In the present paper, we study the second order asymptotics and the derivative of the sum of bifurcated Krein multipliers by adapting the argument in \cite{KLC1}. Our main results are the following two theorems.

\begin{theorem}\label{thm: sum of eig t}
 Consider the solution of \eqref{eq: ham system R modified initial cont pos}. Suppose $\lambda_0=e^{\sqrt{-1}\theta_0}\in U$ is an eigenvalue of $\gamma(0)$. Suppose that $\lambda_0\neq\pm 1$. Assume that
\begin{equation}\label{eq: defn of eta1 eta2}
 \gamma(0)(\eta_1,\eta_2,\bar{\eta}_1,\bar{\eta}_2)=(\eta_1,\eta_2,\bar{\eta}_1,\bar{\eta}_2)\begin{bmatrix}
 \lambda_0 & \lambda_0 & 0 & 0\\
 0 & \lambda_0 & 0 & 0\\
 0 & 0 & \bar{\lambda}_0 & \bar{\lambda}_0\\
 0 & 0 & 0 & \bar{\lambda}_0
 \end{bmatrix}
\end{equation}
and that $\langle A(0)\eta_1,\eta_1\rangle\neq 0$.\footnote{The existence of $\eta_1$ and $\eta_2$ satisfying \eqref{eq: defn of eta1 eta2} is equivalent to the assumption that the geometric multiplicity of $\lambda_0$ is $1$.} Denote by $\lambda_1(t)$ and $\lambda_2(t)$ the eigenvalues of $\gamma(t)$ bifurcated from $\lambda_0$. Then, for $j=1,2$, we have that
 \begin{align}\label{eq: lambdajtsecondexpansion}
  \lambda_j(t)=&\lambda_0+(-1)^{j}\lambda_0\cdot\sqrt{\frac{\langle A(0)\eta_1,\eta_1\rangle}{\langle\eta_2,J_4\eta_1\rangle}t}\notag\\
  &+
  \frac{\lambda_0}{2}\left(\frac{\langle A(0)\eta_1,\eta_1\rangle}{\langle\eta_2,J_4\eta_1\rangle}+\frac{\langle A(0)\eta_1,\eta_2\rangle}{\langle\eta_1,J_{4}\eta_2\rangle}+\frac{\langle A(0)\eta_2,\eta_1\rangle}{\langle\eta_2,J_{4}\eta_1\rangle}-\frac{\langle A(0)\eta_1,\eta_1\rangle\langle\eta_2,J_4\eta_2\rangle}{\langle\eta_2,J_{4}\eta_1\rangle\langle\eta_1,J_{4}\eta_2\rangle}\right)t+o(t).
 \end{align}
 Consequently, we have that
 \begin{equation}\label{eq: real part sum of bifurcated eigenvalues}
  \mathrm{Re} \left(\bar\lambda_0\frac{\mathrm{d}}{\mathrm{d}t}(\lambda_1(t)+\lambda_2(t))|_{t=0}\right)=\frac{\langle A(0)\eta_1,\eta_1\rangle}{\langle\eta_2,J_4\eta_1\rangle}.
 \end{equation}
\end{theorem}

\begin{theorem}\label{thm: sum of eig epsilon}
 Consider the solution of \eqref{eq: perturbed ham system R}. Suppose $\lambda_0=e^{\sqrt{-1}\theta_0}\in U$ is an eigenvalue of $\gamma(T,0)$. Suppose that $\lambda_0\neq\pm 1$. Assume that
\begin{equation}
 \gamma(T,0)(\eta_1,\eta_2,\bar{\eta}_1,\bar{\eta}_2)=(\eta_1,\eta_2,\bar{\eta}_1,\bar{\eta}_2)\begin{bmatrix}
 \lambda_0 & \lambda_0 & 0 & 0\\
 0 & \lambda_0 & 0 & 0\\
 0 & 0 & \bar{\lambda}_0 & \bar{\lambda}_0\\
 0 & 0 & 0 & \bar{\lambda}_0
 \end{bmatrix}.
\end{equation}
Define $\eta_1(t)=\gamma(t,0)\eta_1$ and $\eta_2(t)=\gamma(t,0)\eta_2$
and that $\int_{0}^{T}\langle \frac{\partial}{\partial\varepsilon}A(t,0)\eta_1(t),\eta_1(t)\rangle\,\mathrm{d}t\neq 0$. Let $\varepsilon\geq 0$. Denote by $\lambda_1(\varepsilon)$ and $\lambda_2(\varepsilon)$ the eigenvalues of $\gamma(T,\varepsilon)$ bifurcated from $\lambda_0$ as $\varepsilon$ varies. Then, for $j=1,2$, we have that
 \begin{align}\label{eq: lambdajvarepsilonsecondexpansion}
  \lambda_j(\varepsilon)=&\lambda_0+(-1)^{j}\lambda_0\cdot\sqrt{\int_{0}^{T}\frac{\langle \frac{\partial}{\partial\varepsilon}A(t,0)\eta_1(t),\eta_1(t)\rangle}{\langle\eta_2,J_4\eta_1\rangle}\,\mathrm{d}t\cdot\varepsilon}\notag\\
  &+
  \frac{\lambda_0}{2}\left(-\int_{0}^{T}\frac{\langle \frac{\partial}{\partial\varepsilon}A(t,0)\eta_1(t),\eta_1(t)\rangle}{\langle\eta_2,J_4\eta_1\rangle}\,\mathrm{d}t+\int_{0}^{T}\frac{\langle \frac{\partial}{\partial\varepsilon}A(t,0)\eta_1(t),\eta_2(t)\rangle}{\langle\eta_1,J_{4}\eta_2\rangle}\,\mathrm{d}t\right.\notag\\
  &\left.+\int_{0}^{T}\frac{\langle \frac{\partial}{\partial\varepsilon}A(t,0)\eta_2(t),\eta_1(t)\rangle}{\langle\eta_2,J_{4}\eta_1\rangle}\,\mathrm{d}t-\int_{0}^{T}\frac{\langle \frac{\partial}{\partial\varepsilon}A(t,0)\eta_1(t),\eta_1(t)\rangle\langle\eta_2,J_4\eta_2\rangle}{\langle\eta_2,J_{4}\eta_1\rangle\langle\eta_1,J_{4}\eta_2\rangle}\,\mathrm{d}t\right)\varepsilon+o(\varepsilon).
 \end{align}
 Consequently, we have that
 \begin{equation}
  \mathrm{Re}\left(\bar{\lambda}_0\frac{\mathrm{d}}{\mathrm{d}\varepsilon}(\lambda_1(\varepsilon)+\lambda_2(\varepsilon))|_{\varepsilon=0}\right)=\int_{0}^{T}\frac{\langle \frac{\partial}{\partial\varepsilon}A(t,0)\eta_1(t),\eta_1(t)\rangle}{\langle\eta_2,J_4\eta_1\rangle}\,\mathrm{d}t.
 \end{equation}
\end{theorem}
Our study on the second order asymptotics of bifurcated Krein multipliers and the derivative of the sum of bifurcated Krein multipliers was motivated by \cite{KuwamuraYanagidaMR2271499} of Kuwamura and Yanagida. For perturbed linear Hamiltonian equations as \eqref{eq: perturbed ham system R} but with a general form of $J_4$, they point out that the sum of the bifurcated eigenvalues is differentiable, though neither of them is differentiable. However, their expression (1.7) of the derivative of the sum of the eigenvalues seems to be incorrect.

Note that \eqref{eq: real part sum of bifurcated eigenvalues} is useful for studying the strong stability of $\gamma(t)$. See Remark~\ref{rem: determine the position of eigenvalues}.

By using \cite[Lemma~B.1]{KLC1}, Theorem~\ref{thm: sum of eig t} and \ref{thm: sum of eig epsilon} could be generalized to the equations on $\mathbb{R}^{2n}$ if the algebraic multiplicity of $\lambda_0$ is two and the geometric multiplicity is one.

\bigskip

The organization of this paper is as follows. We collect definitions and notations, prepare some useful properties in Section~\ref{sect: preliminaries}. We prove Theorem~\ref{thm: sum of eig t} in Section~\ref{sect: proof of sum of eig t} and Theorem~\ref{thm: sum of eig epsilon} in Section~\ref{sect: proof of sum of eig epsilon}.

\section{Preliminaries}\label{sect: preliminaries}
\subsection{Notations and definitions}\label{subsect: notation}
\begin{itemize}
 \item For a matrix $M$, we denote by $M^{T}$ the transpose of $M$. For a complex matrix $M$, we denote by $M^{*}$ the conjugate transpose of $M$.
 \item For $n\geq 1$, we denote by $\Id_n$ the $n\times n$ identity matrix and define $J_{2n}\overset{\text{def}}{=}\begin{bmatrix}0 & \Id_n\\-\Id_n & 0\end{bmatrix}$. Then, $J_{2n}^{*}=J_{2n}^{T}=-J_{2n}$ and $J_{2n}^{2}=-\Id_n$.
 \item For vectors $v_1,\ldots,v_n$ in a vector space $V$, we denote by $\wedge_{j=1}^{n}v_j$ the exterior product $v_1\wedge v_2\wedge\cdots\wedge v_n$. (Note that $\wedge$ is associative.) We denote by $\Lambda^{n}(V)$ the space of the linear span of all such $\wedge_{j=1}^{n}v_j$.
 \item For $m\geq 1$, the inner product $\langle\cdot,\cdot\rangle$ on $\mathbb{C}^{m}$ is defined by \[\langle x,y\rangle=\sum_{j=1}^{m}x_j\bar{y}_{j}.\]
 \item Denote by $p(\lambda,t)$ the characteristic polynomial of the matrix $\gamma(t)$, i.e.,
 \[p(\lambda,t)=\det(\lambda\cdot\Id-\gamma(t)).\]
\end{itemize}

\subsection{Exterior powers of linear maps}\label{subsect: exterior powers of linear maps}
We recall exterior powers of a linear map $A$ and its relation with its determinant $\det(A)$.

\begin{defn}\label{defn: ext powers of linear maps}
 Let $A:V\to V$ be a linear map on an $n$-dimensional vector space $V$. For $k=0,\ldots,n$, we define the exterior powers $\bigwedge(n,k,A):\Lambda^{n}(V)\to \Lambda^{n}(V)$ as a linear map as follows:
 \[\bigwedge(n,k,A)(v_1\wedge\cdots \wedge v_n)\overset{\mathrm{def}}{=}\sum_{\sigma\in\{0,1\}^{n}:\sum_{i}\sigma_i=k}\wedge^{n}_{i=1}(\sigma_i\cdot Av_i+(1-\sigma_i)\cdot v_i).\]

 Similarly, for two linear maps $A_1,A_2:V\to V$, for two integers $k_1,k_2=0,\ldots,n$, we define the linear map $\bigwedge(n,k_1,k_2,A_1,A_2):\Lambda^{n}(V)\to \Lambda^{n}(V)$ as follows:
\begin{multline}
 \bigwedge(n,k_1,k_2,A_1,A_2)(v_1\wedge\cdots \wedge v_n)\\
 \overset{\mathrm{def}}{=}\sum_{\sigma\in\{0,1,2\}^{n}:\sum_{i}1_{\sigma_i=1}=k_1,\sum_{i}1_{\sigma_i=2}=k_2}\wedge^{n}_{i=1}(1_{\sigma_i=1}\cdot A_1v_i+1_{\sigma_i=2}\cdot A_2v_i+1_{\sigma_i=0}\cdot v_i),
\end{multline}
 Since $\Lambda^{n}(V)$ is $1$-dimensional, we identify the $\bigwedge(n,k,A)$ (or $\bigwedge(n,k_1,k_2,A_1,A_2)$) with the unique scaling factor, which is also denoted by $\bigwedge(n,k,A)$ (or $\bigwedge(n,k_1,k_2,A_1,A_2)$).
\end{defn}

In the above definition, for each vector $v_i$, we choose one from the three linear maps $\Id$, $A_1$ and $A_2$ and apply it to $v_i$. For the assignment of linear maps to the linear basis, the only constraint is that the map $A_1$ occurs $k_1$ many times and the map $A_2$ occurs exactly $k_2$ many times. All these assignments have equal weight.

Note that $\det(A)$ is identified with the linear map $\bigwedge(n,n,A)$ on the $1$-dimensional vector space $\Lambda^{n}(V)$. In particular, for an eigenvalue $\lambda_0$ of the $2n\times 2n$ matrix $\gamma(0)$, we have that
\begin{multline}\label{eq: char polynomial three terms}
 p(\lambda,t)=\det(\lambda\cdot\Id_{2n}-\gamma(t))=\det((\lambda-\lambda_0)\cdot\Id_{2n}+(\lambda_0\cdot\Id_{2n}-\gamma(0))-(\gamma(t)-\gamma(0)))\\
 =\sum_{k=0}^{2n}(\lambda-\lambda_0)^{k}\sum_{k_1+k_2=2n-k,k_1\geq 0,k_2\geq 0}(-1)^{k_2}\cdot\bigwedge(2n,k_1,k_2,\lambda_0\cdot\Id_{2n}-\gamma(0),\gamma(t)-\gamma(0)).
\end{multline}
In the above calculation, we express the determinant by wedge powers of the sum of linear maps $(\lambda-\lambda_0)\cdot\Id_{2n}$, $\lambda_0\cdot\Id_{2n}-\gamma(0)$ and $\gamma(0)-\gamma(t)$, expand it according to distributive law and collect the terms with the same times of occurrence, where $k_1$ is the time of occurrence of $\lambda_0\cdot\Id_{2n}-\gamma(0)$ and $k_2$ is the time of occurrence of $\gamma(0)-\gamma(t)$.

\subsection{Continuity of roots of polynomials}
We will need the following lemma on the continuity of the roots of polynomials as the coefficients vary.
\begin{lemma}\label{lem: continuity of roots of general polynomials}\cite[Lemma~2.1]{KLC1}
 Let $W$ be a neighborhood of $0$. Let $P_{t}(z)=\sum_{j=0}^{n}c_{j}(t)z^{j}$, where $c_{j}(t)\in\mathbb{C}$ and $t\in W$. Suppose that $t\mapsto c_j(t)$ is continuous for $j=0,\ldots,n$ and $t\in W$. Denote by $d(t)$ the degree of the polynomial $P_{t}$. Suppose that $d(t)=n$ for $t\in W\setminus\{0\}$ and $d(0)=m\leq n$. Then, there exist $m$ continuous complex valued functions $z_1,\ldots,z_{m}$ on $W$ and $n-m$ continuous complex valued functions $z_{m+1},\ldots,z_{n}$ on $W\setminus\{0\}$ such that
 \begin{itemize}
  \item for $t\neq 0$, $z_1(t),\ldots,z_n(t)$ are roots of $P_{t}$,
  \item for $t=0$, $z_1(0),\ldots,z_m(0)$ are roots of $P_0$,
  \item for $i=m+1,\ldots,n$, we have that $\lim_{t\to 0}z_i(t)=\infty$.
 \end{itemize}
\end{lemma}
%\begin{proof}
% By assumptions, for $t_0\in W$, $P_{t}(z)\overset{t\to t_0}{\to}P_{t_0}(z)$ uniformly for $z$ on compacts. Hence, for any continuous loop $\Gamma$ avoiding the roots of $P_{t_0}$, for $t$ sufficiently close to $t_0$, $P_{t}$ does not vanish on $\Gamma$ and
% \begin{equation}\label{eq: contour integral converge as epsilon to delta}
%  \lim_{t\to t_0}\frac{1}{2\pi\sqrt{-1}}\int_{\Gamma}\frac{1}{P_{t}(z)}\,\mathrm{d}z=\frac{1}{2\pi\sqrt{-1}}\int_{\Gamma}\frac{1}{P_{t_0}(z)}\,\mathrm{d}z.
% \end{equation}
% Also, note that for a simple loop avoiding the roots of $P_{t}$, $\frac{1}{2\pi\sqrt{-1}}\int_{\Gamma}\frac{1}{P_{t}(z)}\,\mathrm{d}z$ is precisely the number of roots inside the loop. (The interior and exterior region are determined by the orientation of the loop.) Eventually, Lemma~\ref{lem: continuity of roots of general polynomials} holds since \eqref{eq: contour integral converge as epsilon to delta} holds for all continuous loops avoiding the roots of $P_{t_0}$.
%
%\end{proof}

\section{Proof of Theorem~\ref{thm: sum of eig t}}\label{sect: proof of sum of eig t}

We assume $t\geq 0$ in the following proof. The proof for $t\leq 0$ is similar and we left it to the reader. Recall Definition~\ref{defn: ext powers of linear maps} and \eqref{eq: char polynomial three terms}. We expand the characteristic polynomial $p(\lambda,t)=\det(\lambda\cdot\Id_{4}-\gamma(t))$ at $\lambda_0=e^{\sqrt{-1}\theta_0}$:
\begin{equation}\label{defn: ck expansion at lambda0}
 p(\lambda,t)=\sum_{k=0}^{4}c_{k}(t)(\lambda-\lambda_0)^k.
\end{equation}
Since $p(\lambda,t)$ is a monic polynomial in $\lambda$, we have that $c_4(t)\equiv 1$. For $k=0,\ldots,4$, \[c_k(t)=\sum_{k_2=0}^{4-k}(-t)^{k_2}c_{4-k-k_2,k_2}(t),\]
where
\[c_{k_1,k_2}(t)=\bigwedge(4,k_1,k_2,\lambda_0\cdot\Id-\gamma(0),\frac{1}{t}(\gamma(t)-\gamma(0))).\]
Note that $c_{k,0}(t)$ doesn't depend on $t$ for $k=0,1,\ldots,4$. For simplicity, we denote it by $c_{k,0}$. Since $t\mapsto\gamma(t)$ is $C^2$, we have that $\frac{1}{t}(\gamma(t)-\gamma(0))=\dot{\gamma}(0)+O(t)$ as $t\to 0$. Hence, for $k=0,\ldots,3$,
\begin{equation}\label{eq: ckt expansion Ot2}
 c_{k}(t)=c_{4-k,0}-tc_{3-k,1}(0)+O(t^2).
\end{equation}
In particular, $c_k(0)=c_{4-k,0}$. Since the four eigenvalues of $\gamma(0)$ are $\lambda_0$, $\lambda_0$, $\bar{\lambda_0}$, $\bar{\lambda}_0$, $p(\lambda,0)=\det(\lambda\cdot\Id_4-\gamma(0))=(\lambda-\lambda_0)^2(\lambda-\bar{\lambda}_0)^2=(\lambda-\lambda_0)^4+2(\lambda_0-\bar{\lambda}_0)(\lambda-\lambda_0)^3+(\lambda_0-\bar{\lambda}_0)^2(\lambda-\lambda_0)^2$. Hence,
\begin{equation}\label{eq: c40c30}
 c_{4,0}=c_0(0)=0,\quad c_{3,0}=c_1(0)=0,
\end{equation}
\begin{equation}\label{eq: c20}
 c_{2,0}=c_2(0)=(\lambda_0-\bar{\lambda}_0)^2,
\end{equation}
\begin{equation}\label{eq: c10}
 c_{1,0}=c_3(0)=2(\lambda_0-\bar{\lambda}_0).
\end{equation}

Define $q(\mu,t)=p(\lambda_0+at^{1/2}+\mu t,t)=\sum_{k=0}^4c_k(t)(at^{\frac{1}{2}}+\mu t)^{k}$, where $a$ is chosen such that
\begin{equation}\label{eq: defn of a}
 c_{2,0}a^2=c_{3,1}(0).
\end{equation}
Later, we will see that $a\neq 0$ when $\langle A(0)\xi_1,\xi_1\rangle\neq 0$, see \eqref{eq: a2} below. By \eqref{eq: ckt expansion Ot2} and \eqref{eq: c40c30}, uniformly for $|\mu|\leq 2$, we have that
\[c_4(t)(at^{\frac{1}{2}}+\mu t)^4=O(t^2),\]
\[c_3(t)(at^{\frac{1}{2}}+\mu t)^3=(c_{1,0}+O(t))(a^3t^{\frac{3}{2}}+O(t^2))=c_{1,0}a^3t^{\frac{3}{2}}+O(t^2),\]
\[c_2(t)(at^{\frac{1}{2}}+\mu t)^2=(c_{2,0}+O(t))(a^2t+2a\mu t^{\frac{3}{2}}+O(t^2))=c_{2,0}a^2t+2c_{2,0}a\mu t^{\frac{3}{2}}+O(t^2),\]
\[c_1(t)(at^{\frac{1}{2}}+\mu t)=(-tc_{2,1}(0)+O(t^2))(at^{\frac{1}{2}}+O(t))=-c_{2,1}(0)at^{\frac{3}{2}}+O(t^2),\]
\[c_0(t)=-tc_{3,1}(0)+O(t^2).\]
Hence, $q(\mu,t)=at^{\frac{3}{2}}(2c_{2,0}\mu+c_{1,0}a^2-c_{2,1}(0))+O(t^2)$. Define $r(\mu,t)=\frac{q(\mu,t)}{at^{\frac{3}{2}}}$. Then, $r$ is a polynomial in $\mu$ of degree $4$ for $t>0$. Define $r(\mu,0)=\lim_{t\downarrow 0}r(\mu,t)=2c_{2,0}\mu+c_{1,0}a^2-c_{2,1}(0)$. Then, by Lemma~\ref{lem: continuity of roots of general polynomials}, there exists a unique solution $\mu(t)$ of $r(\mu,t)=0$ such that $\lim_{t\to 0}\mu(t)=\frac{c_{2,1}(0)-c_{1,0}a^2}{2c_{2,0}}$. Thus, among the eigenvalues of $\gamma(t)$, there exists a unique $\lambda_1(t)$ such that for $t\geq 0$, \begin{equation}\label{eq: lambda1t first step}
 \lambda_1(t)=\lambda_0+at^{\frac{1}{2}}+\frac{c_{2,1}(0)-c_{1,0}a^2}{2c_{2,0}}t+o(t).
\end{equation}
Similarly, there exists a unique eigenvalue $\lambda_2(t)$ such that
\begin{equation}\label{eq: lambda2t first step}
 \lambda_2(t)=\lambda_0-at^{\frac{1}{2}}+\frac{c_{2,1}(0)-c_{1,0}a^2}{2c_{2,0}}t+o(t).
\end{equation}

To obtain \eqref{eq: lambdajtsecondexpansion}, it remains to express the above quantities via $\lambda_0$, the generalized eigenvectors and $A(0)$. Assume that
\begin{equation}
 \gamma(0)(\eta_1,\eta_2,\bar{\eta}_1,\bar{\eta}_2)=(\eta_1,\eta_2,\bar{\eta}_1,\bar{\eta}_2)\begin{bmatrix}
 \lambda_0 & \lambda_0 & 0 & 0\\
 0 & \lambda_0 & 0 & 0\\
 0 & 0 & \bar{\lambda}_0 & \bar{\lambda}_0\\
 0 & 0 & 0 & \bar{\lambda}_0
 \end{bmatrix}.
\end{equation}
Then, we have
\begin{lemma}\label{lem: eta1eta2 Krein form}
$\langle\eta_1,J_4\eta_1\rangle=\langle\eta_1,J_4\bar{\eta}_1\rangle=\langle\eta_1,J_4\bar{\eta}_2\rangle=\langle\eta_2,J_4\bar{\eta}_1\rangle=\langle\eta_2,J_4\bar{\eta}_2\rangle=\langle\bar{\eta}_1,J_4\bar{\eta}_1\rangle=0$.
\end{lemma}
\begin{proof}[Proof of Lemma~\ref{lem: eta1eta2 Krein form}]
 Recall \eqref{eq: defn of eta1 eta2} and the assumption that $\lambda_0\in U$ and $\lambda_0\neq\pm 1$. Since $\gamma(0)$ is symplectic, we have that
 \begin{equation*}
  \langle\eta_2,J_4\eta_1\rangle=\langle\gamma(0)\eta_2,J_4\gamma(0)\eta_1\rangle=\langle\lambda_0\eta_2+\lambda_0\eta_1,J_4(\lambda_0\eta_1)\rangle=\langle\eta_2,J_4\eta_1\rangle+\langle\eta_1,J_4\eta_1\rangle.
 \end{equation*}
 Hence, $\langle\eta_1,J_4\eta_1\rangle=0$. Similarly, $\langle\bar{\eta}_1,J_4\bar{\eta}_1\rangle=0$.

 Since $\gamma(0)$ is symplectic, we have that
 \begin{equation*}
  \langle\eta_1,J_4\bar{\eta}_1\rangle=\langle\gamma(0)\eta_1,J_4\gamma(0)\bar{\eta}_1\rangle=\langle\lambda_0\eta_1,J_4(\bar{\lambda}_0\bar{\eta}_1)\rangle=\lambda_0^2\langle\eta_1,J_4\bar{\eta}_1\rangle.
 \end{equation*}
 Since $\lambda_0\neq\pm 1$, we have that
 \begin{equation}\label{eq: eta1 J4 bar eta1}
  \langle\eta_1,J_4\bar{\eta}_1\rangle=0.
 \end{equation}

 Since $\gamma(0)$ is symplectic and \eqref{eq: eta1 J4 bar eta1}, we have that
 \begin{equation*}
  \langle\eta_2,J_4\bar{\eta}_1\rangle=\langle\gamma(0)\eta_2,J_4\gamma(0)\bar{\eta}_1\rangle=\langle\lambda_0\eta_2+\lambda_0\eta_1,J_4(\bar{\lambda}_0\bar{\eta}_1)\rangle=\lambda_0^2\langle\eta_2,J_4\bar{\eta}_1\rangle.
 \end{equation*}
 Since $\lambda_0\neq\pm 1$, we have that
 \begin{equation}\label{eq: eta2 J4 bar eta1}
  \langle\eta_2,J_4\bar{\eta}_1\rangle=0.
 \end{equation}

 Similarly, we have that
 \begin{equation}\label{eq: eta1 J4 bar eta2}
  \langle\eta_1,J_4\bar{\eta}_2\rangle=0.
 \end{equation}

 Since $\gamma(0)$ is symplectic, we have that
 \begin{multline*}
  \langle \eta_2,J_4\bar{\eta}_2\rangle=\langle\gamma(0)\eta_2,J_4\gamma(0)\bar{\eta}_2\rangle=\langle\lambda_0\eta_2+\lambda_0\eta_1,J_4(\bar{\lambda_0}\bar{\eta}_2+\bar{\lambda}_0\bar{\eta}_1)\\
  =\lambda_0^2(\langle\eta_2,J_4\bar{\eta}_2\rangle+\langle\eta_1,J_4\bar{\eta}_1\rangle+\langle\eta_2,J_4\bar{\eta}_1\rangle+\langle\eta_1,J_4\bar{\eta}_2\rangle)=\lambda_0^2\langle\eta_2,J_4\bar{\eta}_2\rangle.
 \end{multline*}
 Since $\lambda_0\neq\pm 1$, we have that
 \begin{equation}\label{eq: eta2 J4 bar eta2}
  \langle\eta_2,J_4\bar{\eta}_2\rangle=0.
 \end{equation}
\end{proof}

We have computed $c_{1,0}$ and $c_{2,0}$ and we will compute $c_{3,1}(0)$ and $c_{2,1}(0)$. For simplicity of notation, define $K=\lambda_0\cdot\Id_4-\gamma(0)$. Then, $K\eta_1=0$, $K\eta_2=-\lambda_0\eta_1$, $K\bar{\eta}_1=(\lambda_0-\bar{\lambda}_0)\bar{\eta}_1$, $K\bar{\eta}_2=(\lambda_0-\bar{\lambda}_0)\bar{\eta}_2-\bar{\lambda}_0\bar{\eta}_1$.
Recall Definition~\ref{defn: ext powers of linear maps}. Since $K\eta_1=0$, we have that
 \begin{align}\label{eq: c310 first step}
 c_{3,1}(0)\eta_1\wedge\eta_2\wedge \bar{\eta}_1\wedge \bar{\eta}_2=&\bigwedge(4,3,1,K,\dot{\gamma}(0))\eta_1\wedge\eta_2\wedge \bar{\eta}_1\wedge \bar{\eta}_2\notag\\
 =&\bigwedge(4,3,1,K,J_4A(0)\gamma(0))\eta_1\wedge\eta_2\wedge \bar{\eta}_1\wedge \bar{\eta}_2\notag\\
 =&J_4A(0)\gamma(0)\eta_1\wedge K\eta_2\wedge K\bar{\eta}_1\wedge K\bar{\eta}_2\notag\\
 =&-\lambda_0^2(\lambda_0-\bar{\lambda_0})^2J_4A(0)\eta_1\wedge \eta_1\wedge \bar{\eta}_1\wedge \bar{\eta}_2.
 \end{align}
 Assume that $J_4A(0)\eta_1=\sum_{k=1}^{4}\alpha_k\eta_k$ where $\eta_3=\bar{\eta}_1$ and $\eta_4=\bar{\eta}_2$. By taking inner product with $J_4\eta_1$ and using Lemma~\ref{lem: eta1eta2 Krein form}, we get that
 \[\langle A(0)\eta_1,\eta_1\rangle=\sum_{k=1}^{4}\alpha_k\langle\eta_k,J_4\eta_1\rangle=\alpha_{2}\langle\eta_2,J_{4}\eta_1\rangle.\]
 Hence,
 \begin{equation}
  \alpha_2=\frac{\langle A(0)\eta_1,\eta_1\rangle}{\langle\eta_2,J_{4}\eta_1\rangle},
 \end{equation}
 \begin{equation}\label{eq: JAIdIdId}
 J_4A(0)\eta_1\wedge \eta_1\wedge \bar{\eta}_1\wedge \bar{\eta}_2=-\frac{\langle A(0)\eta_1,\eta_1\rangle}{\langle\eta_2,J_{4}\eta_1\rangle}\eta_1\wedge \eta_2\wedge \bar{\eta}_1\wedge \bar{\eta}_2
 \end{equation}
 and
 \begin{equation}\label{eq: c310}
  c_{3,1}(0)=\lambda_0^2(\lambda_0-\bar{\lambda}_0)^2\frac{\langle A(0)\eta_1,\eta_1\rangle}{\langle\eta_2,J_{4}\eta_1\rangle}.
 \end{equation}
 Hence, by \eqref{eq: c20} and \eqref{eq: c310}, we have that
 \begin{equation}\label{eq: a2}
  a^2=\frac{c_{3,1}(0)}{c_{2,0}}=\lambda_0^2\frac{\langle A(0)\eta_1,\eta_1\rangle}{\langle\eta_2,J_{4}\eta_1\rangle}.
 \end{equation}
 Similar to the deduction of \eqref{eq: c310 first step}, since $K\eta_1=0$ and $\eta_1\wedge K\eta_2=\eta_1\wedge (-\lambda_0\eta_1)=0$, we have that
 \begin{align}\label{eq: c210}
  c_{2,1}(0)\eta_1\wedge\eta_2\wedge\bar{\eta}_1\wedge\bar{\eta}_2=& J_4A(0)\gamma(0)\eta_1\wedge K\eta_2\wedge K\bar{\eta}_1\wedge \bar{\eta}_2+J_4A(0)\gamma(0)\eta_1\wedge K\eta_2\wedge \bar{\eta}_1\wedge K\bar{\eta}_2\notag\\
  &+J_4A(0)\gamma(0)\eta_1\wedge \eta_2\wedge K\bar{\eta}_1\wedge K\bar{\eta}_2+\eta_1\wedge J_4A(0)\gamma(0)\eta_2\wedge K\bar{\eta}_1\wedge K\bar{\eta}_2.
 \end{align}
 By \eqref{eq: JAIdIdId}, we have that
 \begin{align}\label{eq: JARKKId}
  J_4A(0)\gamma(0)\eta_1\wedge K\eta_2\wedge K\bar{\eta}_1\wedge \bar{\eta}_2=&-\lambda_0^2(\lambda_0-\bar{\lambda}_0)J_4A(0)\eta_1\wedge \eta_1\wedge \bar{\eta}_1\wedge \bar{\eta}_2\notag\\
  =&\lambda_0^2(\lambda_0-\bar{\lambda}_0)\frac{\langle A(0)\eta_1,\eta_1\rangle}{\langle\eta_2,J_{4}\eta_1\rangle}\eta_1\wedge \eta_2\wedge \bar{\eta}_1\wedge \bar{\eta}_2.
 \end{align}
 Similarly, we obtain that
 \begin{equation}\label{eq: JARKIdK}
 J_4A(0)\gamma(0)\eta_1\wedge K\eta_2\wedge \bar{\eta}_1\wedge K\bar{\eta}_2=\lambda_0^2(\lambda_0-\bar{\lambda}_0)\frac{\langle A(0)\eta_1,\eta_1\rangle}{\langle\eta_2,J_{4}\eta_1\rangle}\eta_1\wedge \eta_2\wedge \bar{\eta}_1\wedge \bar{\eta}_2.
 \end{equation}
 Note that
 \begin{equation*}
 J_4A(0)\gamma(0)\eta_1\wedge \eta_2\wedge K\bar{\eta}_1\wedge K\bar{\eta}_2=\lambda_0(\lambda_0-\bar{\lambda}_0)^2J_4A(0)\eta_1\wedge \eta_2\wedge \bar{\eta}_1\wedge \bar{\eta}_2.
 \end{equation*}
 Recall that $J_4A(0)\eta_1=\sum_{k=1}^{4}\alpha_k\eta_k$, where $\alpha_2=\frac{\langle A(0)\eta_1,\eta_1\rangle}{\langle\eta_2,J_{4}\eta_1\rangle}$. By taking inner product with $J_4\eta_2$, we obtain that
 \[\langle A(0)\eta_1,\eta_2\rangle=\alpha_1\langle\eta_1,J_{4}\eta_2\rangle+\alpha_2\langle\eta_2,J_4\eta_2\rangle.\]
 Hence, $\alpha_1=\frac{\langle A(0)\eta_1,\eta_2\rangle}{\langle\eta_1,J_{4}\eta_2\rangle}-\frac{\langle A(0)\eta_1,\eta_1\rangle\langle\eta_2,J_4\eta_2\rangle}{\langle\eta_2,J_{4}\eta_1\rangle\langle\eta_1,J_{4}\eta_2\rangle}$ and
 \begin{equation}\label{eq: JARIdKK}
  J_4A(0)\gamma(0)\eta_1\wedge \eta_2\wedge K\bar{\eta}_1\wedge K\bar{\eta}_2=\lambda_0(\lambda_0-\bar{\lambda}_0)^2\left(\frac{\langle A(0)\eta_1,\eta_2\rangle}{\langle\eta_1,J_{4}\eta_2\rangle}-\frac{\langle A(0)\eta_1,\eta_1\rangle\langle\eta_2,J_4\eta_2\rangle}{\langle\eta_2,J_{4}\eta_1\rangle\langle\eta_1,J_{4}\eta_2\rangle}\right)\eta_1\wedge \eta_2\wedge \bar{\eta}_1\wedge \bar{\eta}_2.
 \end{equation}
 Note that
 \begin{equation}\label{eq: IdJARIdKK1}
 \eta_1\wedge J_4A(0)\gamma(0)\eta_2\wedge K\bar{\eta}_1\wedge K\bar{\eta}_2=\lambda_0(\lambda_0-\bar{\lambda}_0)^2\eta_1\wedge J_4A(0)\eta_2\wedge \bar{\eta}_1\wedge \bar{\eta}_2+\lambda_0(\lambda_0-\bar{\lambda}_0)^2\eta_1\wedge J_4A(0)\eta_1\wedge \bar{\eta}_1\wedge \bar{\eta}_2.
 \end{equation}
 By \eqref{eq: JAIdIdId}, we have that
 \[\eta_1\wedge J_4A(0)\eta_1\wedge \bar{\eta}_1\wedge \bar{\eta}_2=-J_4A(0)\eta_1\wedge\eta_1\wedge \bar{\eta}_1\wedge \bar{\eta}_2=\frac{\langle A(0)\eta_1,\eta_1\rangle}{\langle\eta_2,J_{4}\eta_1\rangle}\eta_1\wedge \eta_1\wedge \bar{\eta}_1\wedge \bar{\eta}_2.\]
 Next, we calculate $\eta_1\wedge J_4A(0)\eta_2\wedge \bar{\eta}_1\wedge \bar{\eta}_2$: Assume $J_4A(0)\eta_2=\sum_{k=1}^{4}\beta_{k}\eta_k$. By taking inner product with $J_{4}\eta_1$, we get that
 \begin{equation*}
  \langle A(0)\eta_2,\eta_1\rangle=\beta_2\langle\eta_2,J_4\eta_1\rangle.
 \end{equation*}
 Hence, we have that $\beta_2=\frac{\langle A(0)\eta_2,\eta_1\rangle}{\langle\eta_2,J_4\eta_1\rangle}$,
 \[\eta_1\wedge J_4A(0)\eta_2\wedge \bar{\eta}_1\wedge \bar{\eta}_2=\frac{\langle A(0)\eta_2,\eta_1\rangle}{\langle\eta_2,J_4\eta_1\rangle}\eta_1\wedge \eta_2\wedge \bar{\eta}_1\wedge \bar{\eta}_2,\]
 and that
 \begin{equation}\label{eq: IdJARIdKK2}
 \eta_1\wedge J_4A(0)\gamma(0)\eta_2\wedge K\bar{\eta}_1\wedge K\bar{\eta}_2=\lambda_0(\lambda_0-\bar{\lambda}_0)^2\left(\frac{\langle A(0)\eta_2,\eta_1\rangle}{\langle\eta_2,J_4\eta_1\rangle}+\frac{\langle A(0)\eta_1,\eta_1\rangle}{\langle\eta_2,J_{4}\eta_1\rangle}\right)\eta_1\wedge \eta_2\wedge \bar{\eta}_1\wedge \bar{\eta}_2.
 \end{equation}
 By \eqref{eq: c210}, \eqref{eq: JARKKId}, \eqref{eq: JARKIdK}, \eqref{eq: JARIdKK} and \eqref{eq: IdJARIdKK2}, we get that
 \begin{align}\label{eq: c210final}
  c_{2,1}(0)=&(2\lambda_0^2(\lambda_0-\bar{\lambda}_0)+\lambda_0(\lambda_0-\bar{\lambda}_0)^2)\frac{\langle A(0)\eta_1,\eta_1\rangle}{\langle\eta_2,J_4\eta_1\rangle}\notag\\
  &+\lambda_0(\lambda_0-\bar{\lambda}_0)^2\left(\frac{\langle A(0)\eta_1,\eta_2\rangle}{\langle\eta_1,J_{4}\eta_2\rangle}+\frac{\langle A(0)\eta_2,\eta_1\rangle}{\langle\eta_2,J_{4}\eta_1\rangle}-\frac{\langle A(0)\eta_1,\eta_1\rangle\langle\eta_2,J_4\eta_2\rangle}{\langle\eta_2,J_{4}\eta_1\rangle\langle\eta_1,J_{4}\eta_2\rangle}\right).
 \end{align}
 By \eqref{eq: c20}, \eqref{eq: c10}, \eqref{eq: lambda1t first step}, \eqref{eq: lambda2t first step}, \eqref{eq: a2} and \eqref{eq: c210final}, we get \eqref{eq: lambdajtsecondexpansion}.

 It remains to prove \eqref{eq: real part sum of bifurcated eigenvalues}.

 By \eqref{eq: lambdajtsecondexpansion}, we have that
 \begin{equation}
  \frac{\mathrm{d}}{\mathrm{d}t}(\lambda_1(t)+\lambda_2(t))|_{t=0}=\lambda_0\left(\frac{\langle A(0)\eta_1,\eta_1\rangle}{\langle\eta_2,J_4\eta_1\rangle}+\frac{\langle A(0)\eta_1,\eta_2\rangle}{\langle\eta_1,J_{4}\eta_2\rangle}+\frac{\langle A(0)\eta_2,\eta_1\rangle}{\langle\eta_2,J_{4}\eta_1\rangle}-\frac{\langle A(0)\eta_1,\eta_1\rangle\langle\eta_2,J_4\eta_2\rangle}{\langle\eta_2,J_{4}\eta_1\rangle\langle\eta_1,J_{4}\eta_2\rangle}\right).
 \end{equation}

 Since $\gamma(0)$ is symplectic, we have that
 \begin{align*}
  \langle\eta_2,J_{4}\eta_2\rangle=&\langle \gamma(0)\eta_2,J_{4}\gamma(0)\eta_2\rangle\\
  =&\langle \lambda_0\eta_2+\lambda_0\eta_1, J_{4}(\lambda_0\eta_2+\lambda_0\eta_1)\rangle\\
  =&\langle\eta_2,J_{4}\eta_2\rangle+\langle\eta_2,J_{4}\eta_1\rangle+\langle\eta_1,J_{4}\eta_2\rangle+\langle\eta_1,J_{4}\eta_1\rangle.
 \end{align*}
 By Lemma~\ref{lem: eta1eta2 Krein form}, $\langle\eta_1,J_{4}\eta_1\rangle=0$. Hence, we have that $\langle\eta_1,J_{4}\eta_2\rangle+\langle\eta_2,J_{4}\eta_1\rangle=0$. Note that $\langle\eta_1,J_{4}\eta_2\rangle=-\langle J_4\eta_1,\eta_2\rangle=-\overline{\langle\eta_2,J_4\eta_1\rangle}$. Hence, $\langle\eta_2,J_{4}\eta_1\rangle$ and $\langle\eta_1,J_{4}\eta_2\rangle$ are real. Therefore, we have that \[\frac{\langle A(0)\eta_1,\eta_2\rangle}{\langle\eta_1,J_{4}\eta_2\rangle}+\frac{\langle A(0)\eta_2,\eta_1\rangle}{\langle\eta_2,J_{4}\eta_1\rangle}=\frac{\langle A(0)\eta_1,\eta_2\rangle}{\langle\eta_1,J_{4}\eta_2\rangle}-\frac{\overline{\langle A(0)\eta_1,\eta_2\rangle}}{\langle\eta_1,J_{4}\eta_2\rangle}=\frac{\langle A(0)\eta_1,\eta_2\rangle}{\langle\eta_1,J_{4}\eta_2\rangle}-\frac{\overline{\langle A(0)\eta_1,\eta_2\rangle}}{\overline{\langle\eta_1,J_{4}\eta_2\rangle}},\]
 which is purely imaginary.

 Since $\langle\eta_2,J_4\eta_2\rangle=-\langle J_{4}\eta_2,\eta_2\rangle=-\overline{\langle\eta_2,J_4\eta_2\rangle}$, the quantity $\langle\eta_2,J_4\eta_2\rangle$ is purely imaginary. Since $\langle A(0)\eta_1,\eta_1\rangle$, $\langle\eta_2,J_{4}\eta_1\rangle$ and $\langle\eta_1,J_{4}\eta_2\rangle$ are real and $\langle\eta_2,J_4\eta_2\rangle$ is imaginary, $-\frac{\langle A(0)\eta_1,\eta_1\rangle\langle\eta_2,J_4\eta_2\rangle}{\langle\eta_2,J_{4}\eta_1\rangle\langle\eta_1,J_{4}\eta_2\rangle}$ is imaginary.

 Since $\langle A(0)\eta_1,\eta_1\rangle$ and $\langle\eta_1,J_{4}\eta_2\rangle$ are real, $\frac{\langle A(0)\eta_1,\eta_1\rangle}{\langle\eta_2,J_4\eta_1\rangle}$ is real.

 Therefore,
 \[\mathrm{Re} \left(\bar{\lambda}_0\frac{\mathrm{d}}{\mathrm{d}t}(\lambda_1(t)+\lambda_2(t))|_{t=0}\right)=\frac{\langle A(0)\eta_1,\eta_1\rangle}{\langle\eta_2,J_4\eta_1\rangle}.\]

 \begin{remark}\label{rem: determine the position of eigenvalues}
  Note that \eqref{eq: real part sum of bifurcated eigenvalues} is useful for studying the strong stability of $\gamma(t)$. A symplectic matrix $\gamma$ is called stable if $\sup_{n\in\mathbb{Z}}||\gamma^n||<\infty$, where $\mathbb{Z}$ denotes the set of integers. It is known that $\gamma$ is stable if it is diagonalizable and all its eigenvalues stay on the unit circle $U\subset\mathbb{C}$. A symplectic matrix $\gamma$ is called strongly stable if there exists a neighborhood of $\gamma$ in the space of symplectic matrices containing only stable symplectic matrices. It is not hard to see that if the four eigenvalues of $\gamma$ are simple and lie on $U$, then $\gamma$ must be strongly stable. In general, the strong stability is related to the Krein type of the eigenvalues. Such a characterization of strong stability was firstly formulated by Krein \cite{KreinMR0036379,KreinMR0043980}, and later independently by Moser \cite{MoserMR0096872}.

  If $\frac{\langle A(0)\eta_1,\eta_1\rangle}{\langle\eta_2,J_4\eta_1\rangle}>0$(resp. or $\frac{\langle A(0)\eta_1,\eta_1\rangle}{\langle\eta_2,J_4\eta_1\rangle}<0$), then there exists $\delta>0$ such that for $t\in(0,\delta)$, $\gamma(t)$ is unstable (resp. strongly stable) and for $t\in(-\delta,0)$, $\gamma(t)$ is strongly stable (resp. unstable).
  \end{remark}

  \begin{proof}[Proof of Remark~\ref{rem: determine the position of eigenvalues}]
  We give the proof for the case $\frac{\langle A(0)\eta_1,\eta_1\rangle}{\langle\eta_2,J_4\eta_1\rangle}>0$. The case $\frac{\langle A(0)\eta_1,\eta_1\rangle}{\langle\eta_2,J_4\eta_1\rangle}<0$ is quite similar.

  Let's first prove that for $t$ positive and sufficiently small, $\gamma(t)$ has four eigenvalues outside of $U$. Hence, such $\gamma(t)$ is unstable. Otherwise, there exists a sequence $\{t_n\}_{n\geq 1}$ strictly decreasing to $0$ such that $\gamma(t_n)$ has four semi-simple eigenvalues $\lambda_1(t_n)$, $\lambda_2(t_n)$, $\bar{\lambda}_1(t_n)$ and $\bar{\lambda}_2(t_n)$ on $U$ and that $\lim_{n\to\infty}\lambda_1(t_n)=\lim_{n\to\infty}\lambda_2(t_n)=\lambda_0$. Then, $\frac{1}{2}(\lambda_1(t_n)+\lambda_2(t_n))$ stay in the unit disk $D\subset \mathbb{C}$. Hence, $\lim_{n\to\infty}\mathrm{Re}\left(\bar{\lambda}_0\frac{\lambda_1(t_n)+\lambda_2(t_n)-2\lambda_0}{t_n}\right)\leq 0$, which contradicts with the assumption $\frac{\langle A(0)\eta_1,\eta_1\rangle}{\langle\eta_2,J_4\eta_1\rangle}>0$ and \eqref{eq: real part sum of bifurcated eigenvalues}.

  Let's show that for $t$ negative and sufficiently small, $\gamma(t)$ has four distinct eigenvalues on the unit circle $U\subset\mathbb{C}$. Hence, such $\gamma(t)$ is strongly stable. Otherwise, there exists a sequence $\{t_n\}_{n\geq 1}$ strictly increasing to $0$ such that $\gamma(t_n)$ has four eigenvalues $\lambda_1(t_n)=r(t_n)e^{\sqrt{-1}\theta(t_n)}$, $\lambda_2(t_n)=\frac{1}{r(t_n)}e^{\sqrt{-1}\theta(t_n)}$, $\bar{\lambda}_1(t_n)=r(t_n)e^{-\sqrt{-1}\theta(t_n)}$ and $\bar{\lambda}_2(t_n)=\frac{1}{r(t_n)}e^{-\sqrt{-1}\theta(t_n)}$ on $U$, where $\lim_{n\to\infty}\theta(t_n)=\theta_0$ and $\lim_{n\to\infty}r(t_n)=1$. (Recall that $\lambda_0=e^{\sqrt{-1}\theta_0}$.) Let $c=\mathrm{Im}\left(\bar{\lambda}_0\frac{\mathrm{d}}{\mathrm{d}t}(\lambda_1(t)+\lambda_2(t))|_{t=0}\right)$. Then, we have that
  \begin{align*}
   ct_n+o(t_n)=&\mathrm{Im}\left(\bar{\lambda}_0(\lambda_1(t_n)+\lambda_2(t_n)-2\lambda_0)\right)\\
   =&\mathrm{Im}\left(\left(r(t_n)+\frac{1}{r(t_n)}\right)e^{\sqrt{-1}(\theta(t_n)-\theta_0)}-2\right)\\
   =&\left(r(t_n)+\frac{1}{r(t_n)}\right)\sin(\theta(t_n)-\theta_0)\\
   =&2(\theta(t_n)-\theta_0)+o(\theta(t_n)-\theta_0).
  \end{align*}
  Hence, $\lim_{n\to\infty}\frac{\theta(t_n)-\theta_0}{t_n}=\frac{c}{2}$. Since $t_n<0$ and $r(t_n)+\frac{1}{r(t_n)}\geq 2$, we have that
  \begin{align*}
   \lim_{n\to\infty}\mathrm{Re}\left(\bar{\lambda}_0\frac{\lambda_1(t_n)+\lambda_2(t_n)-2\lambda_0}{t_n}\right)=&\lim_{n\to\infty}\mathrm{Re}\left(\frac{(r(t_n)+1/r(t_n))e^{\sqrt{-1}(\theta(t_n)-\theta_0)}-2}{t_n}\right)\\
   =&\lim_{n\to\infty}\frac{(r(t_n)+1/r(t_n))\cos(\theta(t_n)-\theta_0)-2}{t_n}\\
   \leq & \lim_{n\to\infty}\frac{2\cos(\theta(t_n)-\theta_0)-2}{t_n}\\
   =&\lim_{n\to\infty}\frac{2\cos(\theta(t_n)-\theta_0)-2}{\theta(t_n)-\theta_0}\cdot\frac{\theta(t_n)-\theta_0}{t_n}\\
   =&0,
  \end{align*}
  which contradicts with the assumption $\frac{\langle A(0)\eta_1,\eta_1\rangle}{\langle\eta_2,J_4\eta_1\rangle}>0$ and \eqref{eq: real part sum of bifurcated eigenvalues}.
 \end{proof}
\section{Proof of Theorem~\ref{thm: sum of eig epsilon}}\label{sect: proof of sum of eig epsilon}

Recall \eqref{defn: BTepsilon}, $\lambda_0\in U$, $\eta_1(t)=\gamma(t,0)\eta_1$ and $\eta_2(t)=\gamma(t,0)\eta_2$. Note that
\begin{align*}
 \langle B(T,0)\eta_1,\eta_2\rangle=& \int_{0}^{T}\left\langle(\gamma(T,0)^{-1})^{T}\gamma(t,0)^{T}\frac{\partial}{\partial\varepsilon}A(t,0)\gamma(t,0)\gamma(T,0)^{-1}\eta_1,\eta_2\right\rangle\,\mathrm{d}t\\
 =&\int_{0}^{T}\left\langle\frac{\partial}{\partial\varepsilon}A(t,0)\gamma(t,0)\gamma(T,0)^{-1}\eta_1,\gamma(t,0)\gamma(T,0)^{-1}\eta_2\right\rangle\,\mathrm{d}t.
\end{align*}
Note that $\gamma(T,0)\eta_1=\lambda_0\eta_1$ and $\gamma(T,0)\eta_2=\lambda_0\eta_2+\lambda_0\eta_1$. Hence, $\gamma(T,0)^{-1}\eta_1=\bar{\lambda}_0\eta_1$ and $\gamma(T,0)^{-1}\eta_2=\bar{\lambda}_0\eta_2-\bar{\lambda}_0\eta_1$. Therefore, we obtain that
\begin{equation}\label{eq: BT0eta1eta2}
 \langle B(T,0)\eta_1,\eta_2\rangle=\int_{0}^{T}\left\langle\frac{\partial}{\partial\varepsilon}A(t,0)\eta_1(t),(\eta_2(t)-\eta_1(t))\right\rangle\,\mathrm{d}t.
\end{equation}
Similarly, we have that
\begin{equation}\label{eq: BT0eta2eta1}
 \langle B(T,0)\eta_2,\eta_1\rangle=\int_{0}^{T}\left\langle\frac{\partial}{\partial\varepsilon}A(t,0)(\eta_2(t)-\eta_1(t)),\eta_1(t)\right\rangle\,\mathrm{d}t.
\end{equation}
and that
\begin{equation}\label{eq: BT0eta1eta1}
 \langle B(T,0)\eta_1,\eta_1\rangle=\int_{0}^{T}\left\langle\frac{\partial}{\partial\varepsilon}A(t,0)\eta_1(t),\eta_1(t)\right\rangle\,\mathrm{d}t.
\end{equation}
The result follows from Theorem~\ref{thm: sum of eig t} and \eqref{eq: BT0eta1eta2}, \eqref{eq: BT0eta2eta1} and \eqref{eq: BT0eta1eta1}.

\bibliographystyle{amsalpha}
\providecommand{\bysame}{\leavevmode\hbox to3em{\hrulefill}\thinspace}
\providecommand{\MR}{\relax\ifhmode\unskip\space\fi MR }
\providecommand{\MRhref}[2]{%
  \href{http://www.ams.org/mathscinet-getitem?mr=#1}{#2}
}
\providecommand{\href}[2]{#2}

\end{document}